\documentclass[10pt,a4paper,oneside,english,reqno]{amsart}
\usepackage[latin9]{inputenc}
\usepackage{babel}
\usepackage{mathrsfs}
\usepackage{amstext}
\usepackage{amsthm}
\usepackage{amssymb}
\usepackage{graphicx}
\usepackage[unicode=true,pdfusetitle,
 bookmarks=true,bookmarksnumbered=false,bookmarksopen=false,
 breaklinks=false,pdfborder={0 0 1},backref=false,colorlinks=false]
 {hyperref}

\makeatletter

\pdfpageheight\paperheight
\pdfpagewidth\paperwidth

\numberwithin{equation}{section}
\numberwithin{figure}{section}
\theoremstyle{plain}
\newtheorem{thm}{\protect\theoremname}
  \theoremstyle{definition}
  \newtheorem*{example*}{\protect\examplename}
  \theoremstyle{definition}
  \newtheorem{defn}{\protect\definitionname}
  \theoremstyle{remark}
  \newtheorem{rem}{\protect\remarkname}
  \theoremstyle{plain}
  \newtheorem{lem}{\protect\lemmaname}

\@ifundefined{date}{}{\date{}}
\usepackage{amsmath}
\usepackage{graphics} 
\usepackage{epsfig} 
\usepackage{graphicx} 
\usepackage{epstopdf}

\hypersetup{urlcolor=blue, citecolor=red}

\makeatother

  \providecommand{\definitionname}{Definition}
  \providecommand{\examplename}{Example}
  \providecommand{\lemmaname}{Lemma}
  \providecommand{\remarkname}{Remark}
\providecommand{\theoremname}{Theorem}

\begin{document}

\title[DACOROGNA-MOSER THEOREM WITH CONTROL OF SUPPORT]{DACOROGNA-MOSER THEOREM ON THE JACOBIAN DETERMINANT EQUATION WITH
CONTROL OF SUPPORT}

\email{pteixeira.ir@gmail.com}

\maketitle
\renewcommand{\thefootnote}{}
\footnote{2010 \emph{Mathematics Subject Classification}. Primary 35F30.}
\footnote{\emph{Key words and phrases}. Volume preserving diffeomorphism, volume correction, prescribed Jacobian PDE, control of support, optimal regularity.}
\renewcommand{\thefootnote}{\arabic{footnote}} \setcounter{footnote}{0}

\centerline{\scshape Pedro Teixeira}
\medskip
{\footnotesize
\centerline{Centro de Matemática da Universidade do Porto}
\centerline{Rua do Campo Alegre, 687, 4169-007 Porto}
\centerline{Portugal}
}

\begin{abstract}
The original proof of Dacorogna-Moser theorem on the prescribed Jacobian
PDE, $\text{det}\,\nabla\varphi=f$, can be modified in order to obtain
control of support of the solutions from that of the initial data,
while keeping optimal regularity. Briefly, under the usual conditions,
a solution diffeomorphism $\varphi$ satisfying
\[
\text{supp}(f-1)\subset\varOmega\Longrightarrow\text{supp}(\varphi-\text{id})\subset\varOmega
\]

\noindent can be found and $\varphi$ is still of class $C^{r+1,\alpha}$
if $f$ is $C^{r,\alpha}$, the domain of $f$ being a bounded connected
open $C^{r+2,\alpha}$ set $\varOmega\subset\mathbb{R}^{n}$.
\end{abstract}

\medskip{}

\emph{\hfill{}In memoriam Jürgen Moser}

\smallskip{}

\section{\textbf{Introduction}}

In \cite[p.4]{DM}, Dacorogna and Moser formulated a celebrated result
on the solutions to the Jacobian determinant PDE with pointwise fixed
boundary condition, which found many applications across several fields
of research. It is one of the main tools for the correction of volume
distortion (in relation to the standard volume) in Hölder spaces.
Its main advantage over similar results lies in its optimal regularity,
the solution diffeomorphism $\varphi$ is $C^{r+1,\alpha}$ if the
initial data $f$ is $C^{r,\alpha}$. Nevertheless, from the point
of view of applications, Dacorogna-Moser theorem has, perhaps, one
main drawback: even if $\text{supp}(f-1)\subset\varOmega$, the solution
obtained does not, in general, extend by the identity to the whole
$\mathbb{R}^{n}$ (in the $C^{r+1,\alpha}$ class). This is a serious
limitation, for it is often necessary to guarantee that the volume
correcting diffeomorphism acts (by composition) only inside the region
$\varOmega$ where the volume distortion takes place, while keeping
the original diffeomorphism (or map) unchanged outside that domain
(see the Example below). This limitation comes from the elliptic regularity
solutions to Neumann problems arising in the proof of the auxiliary
linearized problem. Other approaches (e.g. the flow method of Moser)
permit to take control of support
\begin{equation}
\ensuremath{\text{supp}(f-1)\subset\varOmega}\Longrightarrow\text{supp}(\varphi-\text{id})\subset\varOmega
\end{equation}
but fail to achieve the desired gain of regularity. Notwithstanding,
it is possible to modify Dacorogna-Moser original proof (in its improved
form given in \cite[p.192]{CDK}) in order to guarantee that condition
(1.1) holds, keeping simultaneously optimal regularity, which significantly
enlarges the scope of applicability of the original result.
\begin{thm}
Let $\varOmega\subset\mathbb{R}^{n}$ be a bounded connected open
$C^{r+2,\alpha}$ set, $r\geq0$ an integer and $0<\alpha<1$. Given
$f\in C^{r,\alpha}(\overline{\varOmega})$ such that $f>0$ in $\overline{\varOmega}$
and $\int_{\varOmega}f=\text{\emph{meas}}\,\varOmega$, there exists
$\varphi\in\text{\emph{Diff}}{}^{r+1,\alpha}(\overline{\varOmega},\overline{\varOmega})$
satisfying:
\begin{equation}
\left\{ \begin{array}{lll}
\text{\emph{det}}\,\nabla\varphi=f &  & in\,\,\,\varOmega\\
\varphi=\text{\emph{id}} &  & on\,\,\,\partial\varOmega
\end{array}\right.
\end{equation}
Moreover, if $\text{\emph{supp}}(f-1)\subset\varOmega$ then $\text{\emph{supp}}(\varphi-\text{\emph{id}})\subset\varOmega$
and no regularity needs to be imposed on $\varOmega$.
\end{thm}
\begin{example*}
We give an application to a situation that arises in conservative
dynamics. It corresponds to the natural improvement of the example
given in \cite[p.3]{DM}, made possible by the additional control
of support condition (1.1). Let $r,\alpha$ be as in Theorem 1. Suppose
that $\psi\in\text{Diff}^{r+1,\alpha}(\mathbb{R}^{n})$ and $\varOmega$
is a bounded connected open set such that (a) $\psi(\varOmega)=\varOmega$
and (b) $\psi$ is volume preserving in a neighbourhood of $\mathbb{R}^{n}\setminus\varOmega$
(always in relation to the standard volume; the diffeomorphisms are
of $\mathbb{R}^{n}$ onto itself and orientation preserving). Then
setting $f=\text{det}\,\nabla\psi^{-1}|_{\overline{\varOmega}}$ in
Theorem 1 (noting that $\text{supp}(f-1)\subset\varOmega$), and extending
the solution obtained by the identity to the whole $\mathbb{R}^{n}$,
we find $\varphi\in\text{Diff}^{r+1,\alpha}(\mathbb{R}^{n})$ such
that 
\begin{itemize}
\item $\Psi:=\varphi\circ\psi\in\text{Diff}_{\text{vol}}^{r+1,\alpha}(\mathbb{R}^{n})$
i.e. $\Psi$ is volume preserving on $\mathbb{R}^{n}$
\item $\Psi=\psi$ in a neighbourhood of $\mathbb{R}^{n}\setminus\varOmega$
\end{itemize}
i.e. $\varphi$ corrects the volume distortion of $\psi$ inside $\varOmega$,
while keeping $\psi$ unchanged in $\mathbb{R}^{n}\setminus\varOmega$.
Observe that under the above conditions (a) and (b) plus additional
regularity imposed on $\varOmega$, Dacorogna-Moser theorem only guarantees
the existence of $\varphi_{0}\in\text{Diff}^{r+1,\alpha}(\overline{\varOmega},\overline{\varOmega})$
such that $\Psi_{0}=\varphi_{0}\circ\psi|_{\overline{\varOmega}}\in\text{Diff}_{\text{vol}}^{r+1,\alpha}(\overline{\varOmega},\overline{\varOmega})$
and $\Psi_{0}=\psi$ on $\partial\varOmega$ (prescribed boundary
data), but nothing guarantees that $\Psi_{0}$ extends by $\psi$
to the whole $\mathbb{R}^{n}$ (in the $C^{r+1,\alpha}$ class). Needless
to say, the above reasoning immediately applies to precompact connected
open subsets $\varOmega$ of orientable $n$-manifolds (second countable,
Hausdorff and boundaryless), provided $\varOmega$ smoothly embeds
in $\mathbb{R}^{n}$.
\end{example*}

\subsection{The control of support problem for optimal regularity solutions to
the Jacobian determinant equation.}

The problem of obtaining solutions to the general pullback equation
$f=\varphi^{*}(g)$ (between prescribed volume forms with the same
total volume over a domain), exhibiting both optimal regularity and
control of the support had already been pointed out in \cite[Section III]{DM}
(see also \cite[p.19]{CDK}). Here we will restrict our attention
to the particular case of $g\equiv1$ i.e. to the problem of finding
solutions to the Jacobian determinant equation $\text{det}\,\nabla\varphi=f$
that simultaneously satisfy these two particularly useful conditions.

In \cite[Theorems 3 and 4]{AV}, such solutions were found in the
$C^{\infty}$ case. Standing within $\mathbb{R}^{n}$, Avila used
the duality between divergence-free vector fields and closed $(n-1)$-forms
together with the relative Poincaré lemma and Dacorogna-Moser original
solution \cite[Theorem 2]{DA} to solve the corresponding linearized
problem $\text{div}\,u=f-1$ with control of the support (and, implicitly,
also with that of the norms). Working in the smooth category, he could
then use Moser's flow method \cite[Lemma 3]{DM} to immediately get
the desired solution diffeomorphism. The idea behind the simple yet
efficient and elegant method providing the solution to the linearized
problem seems, as this author himself points out, to have earlier
roots (see for instance \cite[p.290]{TA}). In principle, the same
method could be applied to get the desired solutions to the corresponding
linearized problem in the $C^{r,\alpha}$ (Hölder) case. However,
due to the loss of regularity under exterior derivation, the relative
Poincaré lemma had to be obtained with optimal regularity (see Section
(1.2) below), but such result seemed to be lacking.

In \cite{MA}, Matheus made a simple yet crucial remark: the missing
link (relative Poincaré lemma with optimal regularity, see Theorem
2) could be readily obtained combining the standard relative Poincaré
lemma (with no gain of regularity, see \cite[Theorem 17.3]{CDK},
\cite[p.447]{AMR}) with the quite recent global Poincaré lemma with
optimal regularity of Csató, Dacorogna and Kneuss \cite[p.148]{CDK}.
In fact, this last result turns out to be the key new ingredient in
the solution to the corresponding linearized problem (Theorem 3).
From this point, by following directly the steps of Dacorogna-Moser
original proof, a preliminary solution to the Jacobian determinant
equation with optimal regularity and control of the support in the
case of $\left\Vert f-1\right\Vert _{C^{0,\alpha}}$ small enough
(corresponding to \cite[Lemma 4]{DM}) is readily obtained (Theorem
4), with only a trivial and obvious modification in the definition
of three functional spaces appearing in the original proof (see Section
5 below).

However, to get to the general case, with no restriction imposed on
the Hölder norms of $f-1$, a last difficulty had to be overcome when
adapting the final step in Dacorogna-Moser proof (\cite[Step 4, p.12]{DM}).
Originally, the correction of the measure of an auxiliary function
needed in the process was (implicitly) achieved multiplying it by
a suitable constant (see \cite[p.544]{DA}, \cite[p.201-202 and 391]{CDK}),
but this approach no longer works in the present (control of support)
case, as all functions involved must now equal 1 in a neighbourhood
of $\partial\varOmega$. A new measure correcting method adapted to
this particular case is needed. 

Our contribution here is twofold: to provide this last step (Theorem
6) in the solution to the problem under consideration, thus completing
the puzzle formed by the contributions of Dacorogna-Moser, Avila,
Csató-Dacorogna-Kneuss and Matheus. Due to the action of the convolution
operator, in Theorem 6 the solution diffeomorphism will equal the
identity in a collar that is only slightly thinner than the original
collar where $f$ equals $1$ (this difference being as small as pleased).
The measure correcting problem mentioned above is then solved multiplying
the auxiliary function by a suitable measure correcting smooth function
found via the intermediate value theorem. At this point, Theorem 1
easily follows, any bounded connected open set $\varOmega$ (domain)
having an exhaustion by smooth domains (Appendix, Lemma 1). However,
upon closer inspection, one sees that the control of support conclusion
in Theorem 1 is not completely satisfactory. Assuming that $\varOmega$
is any bounded connected open set and fixing a small enough $d>0$,
for each function $f$ as in the statement and such that the distance
from $\text{supp}(f-1)$ to $\partial\varOmega$ is $\geq d$, by
Theorem 1 there exists a neighbourhood $V_{f}$ of $\partial\varOmega$
(in $\overline{\varOmega})$ where the solution diffeomorphism $\varphi_{f}$
equals the identity, but nothing guarantees a priori the existence
of a neighbourhood $V_{d}$ of $\partial\varOmega$, depending only
on $d$, where \emph{all} these solutions $\varphi_{f}$ (for all
$f$ as above) equal the identity, simultaneously. Actually, this
desirable conclusion easily follows from Theorem 6, and Theorem 7
refines the statement of Theorem 1 to account for it.

Our second aim here is to present a complete and coherent proof of
the whole result, which is roughly sketched in \cite{MA} essentially
only up to Theorem 5. An often neglected condition which proves crucial
when adapting the original proof of Dacorogna-Moser to the control
of support case is \emph{universality}. Care must be taken to ensure
that the bounded linear operator constructed in the solution to the
linearized problem $\text{div}\,u=h$ (Theorem 3) is universal i.e.
independent of $r$ and $\alpha$, otherwise the proof of Theorem
4 (adapting that of \cite[Lemma 4]{DM}) will not work. (For instance,
in \cite[Theorem 3]{AV} it is implicitly assumed that the solution
vector field provided by \cite[Theorem 2]{DM} is smooth ($C^{\infty}$)
if the function $g$ is smooth. Actually this does not follow from
the statement in \cite{DM} (which guarantees only, for each integer
$r\geq0$ and $0<\alpha<1$, the existence of a solution $u_{r,\alpha}$
of class $C^{r+1,\alpha}$), but it does indeed follow from the inspection
of the proof, see Footnote 3 below).

\subsection{Solution to the linearized problem.}

For the convenience of the reader, and serving as a guide to Sections
3 and 4, we detail here part of the proof strategy up to Theorem 3.
The first step in adapting Dacorogna-Moser original proof to the control
of support case is to construct an universal\footnote{See Remark 3.}
bounded linear operator $h\rightarrow u$ solving, with optimal regularity,
the linearized problem

\begin{equation}
\begin{cases}
\text{div\,}u=h & \text{in \ensuremath{\varOmega} }\\
u=0 & \text{in \ensuremath{U}}
\end{cases}
\end{equation}
where $\varOmega\subset\mathbb{R}^{n}$ is a bounded connected open
$C^{\infty}$ set (here briefly called a smooth domain), $h\in C^{r,\alpha}(\overline{\varOmega})$
satisfies $\int_{\varOmega}h=0$ and $h=0$ in a smooth collar $U$
of $\overline{\varOmega}$ (see Definition 2). This is achieved as
follows (c.f. \cite[Theorem 3]{AV}): under the hypothesis above,
\cite[Theorem 2]{DM} provides a first solution $u_{0}\in C^{r+1,\alpha}(\overline{\varOmega})$
to $\text{div\,}u_{0}=h$ in $\varOmega$ which, however, only guarantees
that $u_{0}=0$ on $\partial\varOmega$. But $u_{0}$ can be modified
to satisfy $u=0$ in $U$ while still verifying the estimate
\begin{equation}
\left\Vert u\right\Vert _{C^{r+1,\alpha}}\leq C\left\Vert h\right\Vert _{C^{r,\alpha}}
\end{equation}
for some constant $C=C(r,\alpha,U,\varOmega)>0$ and furthermore,
the correspondence $h\rightarrow u$ can be made linear and universal
(Theorem 3). It is enough to find an universal bounded linear operator
extending $u_{0}|_{U}$ to a divergence-free $C^{r+1,\alpha}$ vector
field $\widetilde{u_{0}}$ on $\overline{\varOmega}$ and then set
$u:=u_{0}-\widetilde{u_{0}}$. Here the idea is to replicate the procedure
in \cite[Theorem 3]{AV}, but now using the relative Poincaré lemma
with optimal regularity (Theorem 2) to circumvent the loss of regularity
under exterior derivation. Briefly, this goes as follows: instead
of trying to extend $u_{0}|_{U}$ directly, in a divergence-free way,
to the whole $\overline{\varOmega}$ (which doesn't seem easy), the
idea is to exploit the natural duality $v\leftrightarrow v\,\lrcorner\,\omega$
(where $\omega$ stands for the canonical volume form on $\mathbb{R}^{n}$)
to make $u_{0}|_{U}$ correspond to a closed $C^{r+1,\alpha}$ $(n-1)$-form
$u^{*}$ on $U$. Since $u_{0}|_{U}$ vanishes on $\partial\varOmega$,
so does $u^{*}$. By the relative Poincaré lemma with optimal regularity
(Theorem 2), there is a $C^{r+2,\alpha}$ $(n-2)$-form $\gamma$
on $U$ such that $d\gamma=u^{*}$. It remains to extend $\gamma$
to the whole $\overline{\varOmega}$ and then follow the inverse procedure
to recover the vector field $u_{0}|_{U}$, now extended (in the $C^{r+1,\alpha}$
class) to the whole $\overline{\varOmega}$, in a divergence-free
way. However, to get the desired solution to (1.3) satisfying (1.4),
care must be taken in the construction of this extension operator
to guarantee that it is, not only linear bounded (in relation to the
$C^{r+2,\alpha}$ norm), but also universal i.e. independent of $r$
and $\alpha$ (as mentioned in Section (1.1), this turns out to be
crucial in the proof of Theorem 4).

\subsection{Limitations of the present solution to the main problem.}

As it will be seen ahead, when $\text{supp}(f-1)\subset\varOmega$,
the construction provided here of a diffeomorphism $\varphi$ simultaneously
satisfying (1.2) and (1.1) depends, in an essentially way, on the
distance $d$ from $\text{supp}(f-1)$ to $\partial\varOmega$. As
a consequence, global estimates of the kind 
\begin{equation}
\left\Vert \varphi-\text{id}\right\Vert _{C^{r+1,\alpha}}\leq C\left\Vert f-1\right\Vert _{C^{r,\alpha}}
\end{equation}
obtained in \cite[p.192]{CDK}, with $C$ is independent of $d$,
which are valid if condition (1.1) is dropped, are actually impossible
to attain by the present method (see Section 8; c.f. Theorem 4). For
this reason, a more uniform method of construction of the solutions
to problem under consideration, permitting useful estimates as (1.6)
with $C$ independent of $d$, would be desirable.

On the other hand, while the existence of optimal regularity solutions
for the more general pullback equation $f=\varphi^{*}(g)$ (between
prescribed volume forms with equal total volume over $\varOmega$)
follows immediately from Dacorogna-Moser result mentioned above (see
\cite[p.4]{DM}), presently we are unable to derive the corresponding
control of support condition
\begin{equation}
\ensuremath{\text{supp}(f-g)\subset\varOmega}\Longrightarrow\text{supp}(\varphi-\text{id})\subset\varOmega
\end{equation}
from the results here obtained.

From the technical point of view, it should be recognized that the
present note adds little to the deepness and beauty of Dacorogna and
Moser's original proof, its main advantage being, perhaps, the complete
transparency and the low deductive effort from previously known results.
It is also worth mentioning that most of the key results involved
in the simple deduction chain that follows are still due to the original
authors, Bernard Dacorogna and Jürgen Moser, working together, alone
or with other authors. This note is dedicated to the memory of the
later.

Finally, we would like to call the reader's attention to the excellent
book \cite{CDK} by Csató, Dacorogna and Kneuss, providing an invaluable
(and quite unique) reference on the pullback equation for differential
forms, the generalized Poincaré lemma and on Hölder spaces in general.

\section{\textbf{Dimension one, notation and conventions}}

\subsection{The one-dimensional case}

When $n=1$ i.e. $\mathbb{R}^{n}=\mathbb{R}$, Dacorogna-Moser Theorem
1' \cite[p.4]{DM} with additional control of support is trivially
true: let $\varOmega=(a,b)$, where $-\infty<a<b<\infty$; let $r\geq0$
be an integer and $0\leq\alpha\leq1$. If $f\in C^{r,\alpha}(\overline{\varOmega})$
satisfies $f>0$ in $\overline{\varOmega}$ and $\int_{a}^{b}f=b-a$,
then
\[
\varphi(x)=a+\int_{a}^{x}f(t)\,dt
\]
belongs to $\text{Diff}^{r+1,\alpha}(\overline{\varOmega},\overline{\varOmega}$),
$\varphi=\text{id on \ensuremath{\partial\varOmega=\{a,b\}}}$ and
for any $0<\eta\leq(b-a)/2$, letting $U=[a,a+\eta]\cup[b-\eta,b],$
it is immediate to verify that
\[
f=1\text{ \,in }U\Longrightarrow\varphi=\text{id\text{ \,in }}U
\]
For this reason we shall concentrate on the case $n\geq2$. Note,
for instance, that while Theorem 3 is trivially true for $n=1$, its
proof fails in that dimension. 

\subsection{Notation and conventions}

For brevity of expression, we introduce the following definition of
domain, which is narrower than the usual one (as boundedness is imposed).
\begin{defn}
\emph{$\bullet$~Domain}. A bounded, connected open set $\varOmega\subset\mathbb{R}^{n}$
is called here a \emph{domain}. Let $r\geq0$ be an integer and $0\leq\alpha\leq1$.
Domains with $C^{r,\alpha}$ boundary (briefly $C^{r,\alpha}$ domains)
are defined in the usual way \cite[p.338]{CDK}. A domain is smooth
if it is $C^{\infty}$. 

\emph{$\bullet$~Banach space} $C^{r,\alpha}(\overline{\varOmega})$.
Let $\varOmega\subset\mathbb{R}^{n}$ be a domain, $r\geq0$ an integer
and $0\leq\alpha\leq1$. The space $C^{r,\alpha}(\varOmega)$ is defined
in the usual way \cite[p.336-337]{CDK}. $C^{0}(\overline{\varOmega})$
is the space of continuous functions on $\overline{\varOmega}$ and
we define $C^{r,\alpha}(\overline{\varOmega})$ as the subspace of
all functions $f\in C^{0}(\overline{\varOmega})$ such that (1) $f|_{\varOmega}$
is $C^{r}$, (2) all its partial derivatives up to order $r$ extend
continuously to $\overline{\varOmega}$ and (3) for every multiindex
$b$ of order $|b|=r$, $[\partial^{b}f]_{C^{0,\alpha}(\overline{\varOmega})}<\infty,$
where for $D\subset\mathbb{R}^{n}$ (with more than one point),
\[
[g]_{C^{0,\alpha}(D)}:=\underset{x,y\in D;\,x\neq y}{\text{sup}}\left\{ \frac{\left|g(x)-g(y)\right|}{\left|x-y\right|^{\alpha}}\right\} 
\]
As usual, $u\in C^{r,\alpha}(\overline{\varOmega};\mathbb{R}^{n})$
if all its components belong to $C^{r,\alpha}(\overline{\varOmega}).$
If $\varOmega,\varOmega'\subset\mathbb{R}^{n}$ are domains and $r\geq1$,
then $\varphi\in\text{Diff}^{r,\alpha}(\overline{\varOmega},\overline{\varOmega'})$
if $\varphi$ is a bijection from $\overline{\varOmega}$ to $\overline{\varOmega'}$
and $\varphi\in C^{r,\alpha}(\overline{\varOmega};\mathbb{R}^{n})$,
$\varphi^{-1}\in C^{r,\alpha}(\overline{\varOmega'};\mathbb{R}^{n})$.
\end{defn}
\noindent \textbf{Remark on notation. }Our definition of $C^{r,\alpha}(\overline{\varOmega})$
is slightly different from the more common one, which usually defines
$C^{r,\alpha}(\overline{\varOmega})$ as consisting of all restrictions
$f|_{\varOmega}$, where $f\in C^{0}(\overline{\varOmega})$ satisfies
conditions (1) to (3) above. The present definition is more convenient
in the following sense: if $f\in C^{r,\alpha}(\overline{\varOmega})$
and \textbf{$\varOmega$ }is\textbf{ }Lipschitz (i.e. $C^{0,1}$),
then $f$ has a $C^{r,\alpha}$-extension $\widetilde{f}$ to the
whole $\mathbb{R}^{n}$ (see  \cite[p.342]{CDK}).\footnote{\noindent All domains we will encounter are at least $C^{2,\alpha}$.}
By $C^{r}$ continuity, all the partial derivatives of $\widetilde{f}$
up to order $r$ at the points of $\partial\varOmega$ (i.e. the $r$-jets
of $f|_{\partial\varOmega}$) are uniquely determined by $f|_{\varOmega}$,
and therefore they are common to all possible $C^{r,\alpha}$-extensions
$\widetilde{f}$. Hence, it makes sense to evaluate all these partial
derivatives of $f$ on $\partial\varOmega$ and $f$ should be regarded
as the restriction to $\overline{\varOmega}$ of all possible $C^{r,\alpha}$
extensions of $f$ to the whole $\mathbb{R}^{n}$. If $u\in C^{r,\alpha}(\overline{\varOmega};\mathbb{R}^{n})$
and $r\geq1$, then the existence of these extensions for $u$ also
implies, for instance, that $\text{div }u$ and $\text{det}\,\nabla u$,
which are defined in $\varOmega$, have $C^{r-1,\alpha}$-extensions
to the whole $\mathbb{R}^{n}$, which are uniquely determined on $\partial\varOmega$
by $u$. Therefore, $\text{div }u$ and $\text{det}\,\nabla u$ should
also be seen as belonging to $C^{r-1,\alpha}(\overline{\varOmega})$.
In this context, we see that in Dacorogna-Moser theorem \cite[Theorem 1']{DM},
where $\varOmega$ is $C^{r+3,\alpha}$ and thus Lipschitz, the solution
$\varphi\in\text{Diff}^{r+1,\alpha}(\overline{\varOmega},\overline{\varOmega})$
actually satisfies

\[
\text{det\,}\nabla\varphi=f\text{ \,\,\,\,\,in \ensuremath{\overline{\varOmega}} }
\]
or simply
\[
\text{det\,}\nabla\varphi=f
\]
as both $\text{det\,}\nabla\varphi$ and $f$ are in $C^{0}(\overline{\varOmega})$
and $\text{det\,}\nabla\varphi=f$ in $\varOmega$. The above identity
actually means that $\text{det\,}\nabla\varphi$ and $f$ have $r$-jet
coincidence all over $\ensuremath{\overline{\varOmega}}$ (by $C^{r}$
continuity, these $r$-jets still coincide on $\partial\varOmega$).
Analogously, in \cite[Theorem 2]{DM} we actually have 
\[
\text{div\,}u=h\,\,\,\,\,\text{(in }\ensuremath{\overline{\varOmega}}\text{)}
\]
For brevity, we shall adopt from now on this natural convention: $f=g$
means that these two functions have the same domain and agree all
over it. 

The present definition of $C^{r,\alpha}(\overline{\varOmega})$ is
more consistent than the usual one, since adopting the later it is
still often necessary to evaluate $f\in C^{r,\alpha}(\overline{\varOmega})$
(and functions depending continuously on its $r$-jet) at points of
$\partial\varOmega$, while the domain of $f$ is actually $\varOmega$,
by definition. Moreover, for $u$ as above, $\partial\varOmega\subset\overline{\varOmega}$
immediately implies $\left\Vert u\right\Vert _{C^{r,\alpha}}=\left\Vert u\right\Vert _{C^{r,\alpha}(\varOmega)}:=\left\Vert u|_{\varOmega}\right\Vert _{C^{r,\alpha}}$,
where $\left\Vert u\right\Vert _{C^{r,\alpha}}=\left\Vert u\right\Vert _{C^{r}}+\underset{|b|=r}{\text{max}}\,[\partial^{b}u]_{C^{0,\alpha}(\overline{\varOmega})}$
(see \cite[p.336]{CDK}).

\section{\textbf{Optimal regularity relative Poincaré lemma}}
\begin{defn}
(\emph{Collar of} $\overline{\varOmega}$). If $\varOmega\subset\mathbb{R}^{n}$
is a smooth domain, there is a smooth embedding $\zeta:\partial\varOmega\times[0,\infty)\hookrightarrow\overline{\varOmega}$
such that $\zeta(x,0)=x$ (collar embedding \cite[Chapter 4]{HI}).
For each $\epsilon>0$ we call $U_{\epsilon}:=\zeta(\partial\varOmega\times[0,\epsilon${]})
a (compact)\emph{ collar of} $\overline{\varOmega}.$ Every neighbourhood
of $\partial\varOmega$ contains a collar and every collar is contained
in the relative interior of another collar.\emph{ }
\end{defn}
The following result is the key lemma in this note and it is interesting
on its own.
\begin{thm}
\emph{(Optimal regularity relative Poincaré lemma)}. Let $r\geq1$
and $1\leq k\leq n$ be integers and $0<\alpha<1$. Let $\varOmega\subset\mathbb{R}^{n}$,
$n\geq2$, be a bounded connected open smooth set and $U$ a collar
of $\overline{\varOmega}$. Then there is a constant $C=C(r,\alpha,U)>0$
such that: given a closed form $\beta\in C^{r,\alpha}(U;\varLambda^{k})$
that vanishes when pulledback to $\partial\varOmega$ (i.e. $i^{*}\beta=0$
where $i:\partial\varOmega\rightarrow\overline{\varOmega}$ is the
inclusion), there exists $\omega\in C^{r+1,\alpha}(\overline{\varOmega};\varLambda^{k-1})$
satisfying:
\begin{enumerate}
\item $d\omega=\beta$ in $U$
\item $\left\Vert \omega\right\Vert _{C^{r+1,\alpha}}\leq C\left\Vert \beta\right\Vert _{C^{r,\alpha}(U)}$
\end{enumerate}
Furthermore, the correspondence $\beta\rightarrow\omega$ can be chosen
linear and universal.
\end{thm}
\begin{rem}
(Universality). The above correspondence is \emph{universal }in the
sense that $\omega$ depends only on $\beta$ and $\varOmega$, but
not on $r,\,\alpha$. More precisely, if $\beta$ also belongs to
class $C^{s,\delta}$, $s\in\mathbb{Z}^{+}$ and $0<\delta<1,$ then
the same solution $\omega$ is of class $C^{s+1,\delta}$ and the
corresponding estimate (2) holds for some constant $C=C(s,\delta,U)>0$.
In particular, $\omega$ is $\text{smooth if \ensuremath{\beta} is smooth.}$
The universality of $\beta\rightarrow\omega$ will be used in the
proof of  Theorem 3.
\end{rem}
\begin{proof}
(A) $\beta$ \emph{has a $C^{1}$ primitive $\omega_{0}$ in $U$}
(i.e. $d\omega_{0}=\beta$). 

Since $U$ is a collar, $\beta$ is $C^{1}$, $d\beta=0$ and $\beta$
vanishes when pulledback to $\partial\varOmega$, $\beta$ has a primitive
$\omega_{0}$ of class $C^{1}$ on $U$. This is immediate to verify
following the standard proof of the relative Poincaré lemma, which
uses the homotopy formula with integration along time fibres. Alternatively,
this fact immediately follows from the more general result \cite[Theorem 17.3]{CDK},
taking Remark 17.4 into consideration (briefly, let $\zeta$ be the
map defining the collar and $\zeta^{-1}(x)=(\psi(x),\tau(x))$, where
$\psi(x)\in\partial\varOmega$ and $\tau(x)\in[0,\epsilon]$. Define
the smooth map
\[
\begin{array}{lll}
F_{t}(x)=F(t,x):\,[0,1]\times U & \longrightarrow & U\\
\,\,\,\,\,\,\,\,\,\,\,\,\,\,\,\,\,\,\,\,\,\,\,\,\,\,\,\,\,\,\,\,\,\,\,\,\,\,\,\,\,\,\,\,\,\,\,\,\,\,\,\,\,\,\,(t,\,x) & \longmapsto & \zeta(\psi(x),t\tau(x))
\end{array}
\]
Then, (i) $\beta$ is closed, (ii) $F_{1}^{*}(\beta)=\beta$ since
$F_{1}=\text{id}$ and (iii) $F_{0}^{*}(\beta)=0$ since $F_{0}(U)\subset\partial\varOmega$
and $\beta$ vanishes when pulledback to $\partial\varOmega$. Therefore
(\cite[Theorem 17.3]{CDK}), there exists on $U$ a $(k-1)$-form
$\omega_{0}$ of class $C^{1}$ satisfying $d\omega_{0}=\beta$).

(B) \emph{Finding a $C^{r+1,\alpha}$ primitive $\widehat{\omega}$
of $\beta$ in $U$ with control of the norm. }

Now observe that $\varOmega':=\text{int}\,U$ (in $\mathbb{R}^{n}$)
and $\beta$ satisfy the hypothesis of \cite[Theorem 8.3]{CDK} since
$d\beta=0$ and for every $\psi\in\mathcal{\mathscr{H}}_{N}(\text{int}\,U,\varLambda^{k})$
(noting that each such $\psi$ extends to $\widetilde{\psi}\in C^{\infty}(U,\varLambda^{k})$
\cite[p.122]{CDK}), we have, integrating by parts \cite[p.88]{CDK}
and using the primitive $\omega_{0}$ found in (A), 
\[
\int_{\text{int\,}U}\left\langle \beta,\psi\right\rangle =\int_{\text{int}\,U}\left\langle d\omega_{0},\psi\right\rangle =-\int_{\text{int}\,U}\left\langle \omega_{0},\delta\psi\right\rangle +\int_{\partial U}\left\langle \omega_{0},\nu\lrcorner\psi\right\rangle =0
\]
as $\delta\psi=0=\nu\lrcorner\psi$ by definition of $\mathcal{\mathscr{H}}_{N}$.
Therefore $\beta$ has a $C^{r+1,\alpha}$ primitive $\widehat{\omega}$
on $U$ satisfying
\[
\left\Vert \widehat{\omega}\right\Vert _{C^{r+1,\alpha}(U)}\leq C_{1}\left\Vert \beta\right\Vert _{C^{r,\alpha}(U)}
\]
where $C_{1}=C_{1}(r,\alpha,$$U)>0$ is the constant given in \cite[Theorem 8.3]{CDK}.

(C) \emph{Universal extension of $\widehat{\omega}$ to the whole
$\overline{\varOmega}$ with control of the $C^{r+1,\alpha}$ norm.}
As a $k$-form\emph{ }is completely determined by its $\tbinom{n}{k}$
components and $U$ is smooth, there is a universal linear operator
(see \cite[p.342]{CDK})
\[
E:\,C^{r+1,\alpha}(U;\varLambda^{k-1})\longrightarrow C^{r+1,\alpha}(\overline{\varOmega};\varLambda^{k-1})
\]
and a constant $C_{2}=C_{2}(r,U)\geq1$ such that 
\[
E(\gamma)|_{U}=\gamma\text{\, and \,}\left\Vert E(\gamma)\right\Vert _{C^{r+1,\alpha}}\leq C_{2}\left\Vert \gamma\right\Vert _{C^{r+1,\alpha}(U)}
\]
hence the extension $\omega=E(\widehat{\omega})\in C^{r+1,\alpha}(\overline{\varOmega};\varLambda^{k-1})$
satisfies 
\[
\left\Vert \omega\right\Vert _{C^{r+1,\alpha}}\leq C_{1}C_{2}\left\Vert \beta\right\Vert _{C^{r,\alpha}(U)}
\]
Therefore, $d\omega=\beta$ in $U$ and there is $C=C(r,\alpha,U)>0$
as claimed. Since $\beta\rightarrow\widehat{\omega}$ and $\widehat{\omega}\rightarrow\omega$
are both linear and universal (see \cite[p.148-149]{CDK}) so is $\beta\rightarrow\omega$.
\end{proof}
\begin{rem}
Actually, in the above situation, a much simpler extension operator
could be used. If $h\in C^{r,\alpha}$($\overline{\varOmega}$) and
$\varOmega$ is $C^{r,\alpha}\cap C^{1}$, then $h$ can be $C^{r,\alpha}$-extended
by a bounded linear operator as in \cite[p.136]{GT}. But even if
the domain $\varOmega$ is smooth, this extension operator actually
fails to be universal as $E(\gamma$) depends on $r$. To render it
universal (in the case $\varOmega$ is smooth), instead of balls and
half-balls we use (open) cubes $(-2,2)^{n}$ and halfcubes $[0,2)\times(-2,2)^{n-1}$
for the smooth boundary rectification, and use Seeley's operator \cite{SE}
(c.f. also \cite{BI}) to extend functions from the right halfcube
to the cube, noting that, by construction, Seeley's extension of a
function $h$ to the left halfline $(-\infty,0]\times y$ depends
only on the values taken by $h$ on the interval $[0,2)\times y$.
This operator provides a simultaneous bounded linear extension in
all classes of differentiability.
\end{rem}

\section{\textbf{The linearized problem when $h=0$ in a collar}}
\begin{thm}
Let $r\geq0$ be an integer and $0<\alpha<1$. Let $\varOmega\subset\mathbb{R}^{n}$,
$n\geq2$, be a bounded connected open smooth set and $U$ a collar
of $\overline{\varOmega}$. Given $h\in C^{r,\alpha}(\overline{\varOmega})$
satisfying
\[
\begin{cases}
\int_{\varOmega}h=0 & \text{ }\\
h=0 & \text{in \ensuremath{U}}
\end{cases}
\]
there exists $u\in C^{r+1,\alpha}(\overline{\varOmega};\mathbb{R}^{n})$
satisfying 
\[
\begin{cases}
\text{\emph{div} }u=h & \text{ }\\
u=0 & \text{in }\ensuremath{U}
\end{cases}
\]
Furthermore, the correspondence $h\rightarrow u$ can be chosen linear
and universal and there exists $C=C(r,\alpha,U,\varOmega)>0$ such
that
\[
\left\Vert u\right\Vert _{C^{r+1,\alpha}}\leq C\left\Vert h\right\Vert _{C^{r,\alpha}}
\]
\end{thm}
\begin{rem}
(Universality). The above correspondence is universal\emph{ }in the
sense that $u$ depends only on $h$ and $U$ (to say that $u$ depends
on $\varOmega$ is redundant since $\varOmega=\text{dom}\,h$), but
not on $r,\,\alpha$. More precisely, if $h$ also belongs to class
$C^{s,\delta}$, $s\in\mathbb{N}_{0}$ and $0<\delta<1,$ then the
same solution $u$ is $C^{s+1,\delta}$ and the corresponding estimate
holds for some constant $C=C(s,\delta,U)>0$. In particular, $u$
is $\text{smooth if \ensuremath{h} is smooth.}$ The universality
of $h\rightarrow u$ will be used in the proof of Theorem 4.
\end{rem}
\begin{proof}
(A) \emph{Reduction to existence of a divergence-free extension universal
bounded linear operator.} Theorem 9.2 in \cite[p.180]{CDK} provides
a first solution to the problem, $u_{0}\in C^{r+1,\alpha}(\overline{\varOmega};\mathbb{R}^{n})$,
which however only guarantees that $u_{0}=0$ on $\partial\varOmega$.
Moreover, the correspondence $h\rightarrow u_{0}$ is linear and universal\footnote{When $\varOmega$ is smooth, the correspondence $h\rightarrow u_{0}$
given by \cite[Theorem 9.2]{CDK} can be made universal in the sense
that $u_{0}$ can be made to depend only on $h$ but not on $r$ and
$\alpha$, as made precise in Remark 3. This is easy to verify inspecting
its proof: in Step 1, the Neumann problem with condition $\int_{\varOmega}w=0$
has a unique solution, which guarantees the universality of $h\rightarrow w$.
In Step 2, opting for Proof 2 of Lemma 8.8 \cite[p.150]{CDK} and
taking smooth admissible boundary coordinate systems $\varphi_{i}$,
the correspondence $w\rightarrow v$ is also clearly universal.} and there is a constant $C_{1}=C_{1}(r,\alpha,\varOmega)>0$ such
that
\[
\left\Vert u_{0}\right\Vert _{C^{r+1,\alpha}}\leq C_{1}\left\Vert h\right\Vert _{C^{r,\alpha}}
\]
Since $\text{div}\,u_{0}=h=0$ in $U$ (see Remark on notation, Section
2.2), in order to find $u$ it is enough to construct an universal
(bounded) linear operator $H(\cdot)$ extending each divergence-free
$X\in C^{r+1,\alpha}(U;\mathbb{R}^{n})$ vanishing on $\partial\varOmega$,
to a divergence-free $H(X)\in C^{r+1,\alpha}(\overline{\varOmega};\mathbb{R}^{n})$
such that 
\[
\left\Vert H(X)\right\Vert _{C^{r+1,\alpha}}\leq C_{2}\left\Vert X\right\Vert _{C^{r+1,\alpha}(U)}
\]
for some constant $C_{2}=C_{2}(r,\alpha,U)>0$, for it is then immediate
to verify that $u:=u_{0}-H(u_{0}|_{U})$ is the desired solution and
$C=C(r,\alpha,U,\varOmega)=C_{1}(1+C_{2})$. As $h\rightarrow u_{0}$
and $u_{0}\rightarrow u$ are both linear and universal so is $h\rightarrow u$. 

(B) \emph{Construction of operator $H(\cdot)$.} Suppose that $X\in C^{r+1,\alpha}(U;\mathbb{R}^{n})$
is divergence-free and vanishes on $\partial\varOmega$. In order
to construct $\widetilde{X}:=H(X)$ we use, as in \cite[Theorem 3]{AV},
the isomorphism between $C^{r+1,\alpha}$ divergence-free vector fields
and $C^{r+1,\alpha}$ closed $(n-1)$-forms given by
\[
X\longleftrightarrow X^{*}=X\,\lrcorner\,\omega
\]
where $\omega$ is the standard volume form on $\mathbb{R}^{n}$,
which immediately gives
\begin{itemize}
\item $X^{*}\in C^{r+1,\alpha}(U;\varLambda^{n-1})$
\item $dX^{*}=(\text{div}\,X)\omega=0$
\item $X^{*}=0$ in $\partial\varOmega$ (as $X=0$ there)
\item $\left\Vert X\right\Vert _{C^{r+1,\alpha}(U)}=\left\Vert X^{*}\right\Vert _{C^{r+1,\alpha}(U)}$
\end{itemize}
Applying Theorem 2 to $X^{*}$ we find $\gamma\in C^{r+2,\alpha}(\overline{\varOmega};\varLambda^{n-2})$
and a constant $C_{3}=C_{3}(r,\alpha,U)>0$ such that $d\gamma=X^{*}$
in $U$ and 
\[
\left\Vert \gamma\right\Vert _{C^{r+2,\alpha}}\leq C_{3}\left\Vert X^{*}\right\Vert _{C^{r+1,\alpha}(U)}
\]
Moreover, the correspondence $X^{*}\rightarrow\gamma$ is both linear
and universal. We now go in the opposite direction:
\begin{itemize}
\item $d\gamma\in C^{r+1,\alpha}(\overline{\varOmega};\varLambda^{n-1})$
is a closed form
\item $d\gamma=X^{*}$ in $U$
\item $\left\Vert d\gamma\right\Vert _{C^{r+1,\alpha}}\leq(n-1)\left\Vert \gamma\right\Vert _{C^{r+2,\alpha}}$
\end{itemize}
\noindent and $d\gamma$ corresponds, under the isomorphism described
above, to a divergence-free $\widetilde{X}\in C^{r+1,\alpha}(\overline{\varOmega};\mathbb{R}^{n})$,
$\widetilde{X}|_{U}=X$. The correspondence $X\rightarrow\widetilde{X}$
is linear and universal, being a composition of universal linear operators
\[
X\longrightarrow X^{*}\longrightarrow\gamma\longrightarrow d\gamma\longrightarrow\widetilde{X}=H(X)
\]
Following the above chain we readily get $C_{2}=(n-1)C_{3}$ ($n=\text{dim}\,\varOmega$)
and $H(\cdot)$ is as claimed.
\end{proof}

\section{\textbf{Solution when $f=1$ in a collar and $\left\Vert f-1\right\Vert _{C^{0,\gamma}}$
is small} }

Directly following Dacorogna-Moser original proof, a first solution
is found under the assumption that $\left\Vert f-1\right\Vert _{C^{0,\gamma}}$
is sufficiently small, for some $0<\gamma\leq\alpha$ (actually, under
a slightly more general hypothesis). An advantage of this preliminary
solution is that it automatically gives useful estimates on the Hölder
norms of $\varphi-\text{id}$, and for this reason we rather follow
the proof of Theorem 10.9 in \cite[p.198]{CDK} (which improves Lemma
4 in \cite[p.10]{DM}). The only change needed in that proof essentially
amounts to an obvious and rather trivial modification in the definitions
of the functional spaces $X$, $Y$ and $B$ (all references below
are to \cite[p.198-201]{CDK}). We call the reader's attention to
the fact that the universality of the operator constructed in Theorem
3 above is used in an essential way below. 
\begin{thm}
Let $r\geq0$ be an integer and $0<\alpha,\gamma<1$ with $\gamma\leq r+\alpha$.
Let $\varOmega\subset\mathbb{R}^{n}$, $n\geq2$, be a bounded connected
open smooth set and $U$ a collar of $\overline{\varOmega}$. Then,
there are constants $\epsilon=\epsilon(r,\alpha,\gamma,U,\varOmega)>0$
and $c=c(r,\alpha,\gamma,U,\varOmega)>0$ such that: given $f\in C^{r,\alpha}(\overline{\varOmega})$,
$f>0$ in $\overline{\varOmega}$, satisfying
\[
\begin{cases}
\int_{\varOmega}f=\text{\emph{meas}}\,\varOmega & \text{ }\\
f=1 & \text{in \ensuremath{U}}\\
\left\Vert f-1\right\Vert _{C^{0,\gamma}}\leq\epsilon
\end{cases}
\]
there exists $\varphi\in\text{\emph{Diff}}{}^{r+1,\alpha}(\overline{\varOmega},\overline{\varOmega})$
satisfying
\begin{equation}
\begin{cases}
\text{\emph{det}}\,\nabla\varphi=f & (\text{in \ensuremath{\overline{\varOmega}}})\\
\varphi=\text{\emph{id}} & \text{in }U
\end{cases}
\end{equation}

\[
\left\Vert \varphi-\text{\emph{id}}\right\Vert _{C^{r+1,\alpha}}\leq c\left\Vert f-1\right\Vert _{C^{r,\alpha}}\text{ \,\,and \,\,\,}\left\Vert \varphi-\text{\emph{id}}\right\Vert _{C^{1,\gamma}}\leq c\left\Vert f-1\right\Vert _{C^{0,\gamma}}
\]
\end{thm}
\begin{rem}
See Remark on notation (Section 2) for the identity $\text{det}\,\nabla\varphi=f$
in $\overline{\varOmega}$ above. 
\end{rem}
\begin{rem}
The following fundamental result on the inclusion of Hölder spaces
will be often implicitly used without mention: if the domain $\varOmega\subset\mathbb{R}^{n}$
is Lipschitz, $0\leq\widetilde{s}\leq s$ are integers and $0\leq\beta,\widetilde{\beta}\leq1$,
with $\widetilde{s}+\widetilde{\beta}\leq s+\beta$, then $C^{s,\beta}(\overline{\varOmega})\subset C^{\widetilde{s},\widetilde{\beta}}(\overline{\varOmega})$
\cite[p.342]{CDK}.
\end{rem}
\begin{proof}
\emph{Step 1. }Let
\[
X=\left\{ a\in C^{r+1,\alpha}(\overline{\varOmega};\mathbb{R}^{n}):\,\,\,a=0\text{ \,\,in }U\right\} 
\]

\[
Y=\left\{ b\in C^{r,\alpha}(\overline{\varOmega}):\,\,\,\int_{\varOmega}b=0\text{ and }b=0\text{ \,\,in }U\right\} 
\]
By the divergence theorem, $L(a)=\text{div}\,a$ is a well defined
bounded linear operator $L:X\rightarrow Y$ . By Theorem 3 there is
a bounded linear operator $L^{-1}:Y\rightarrow X$ such that $LL^{-1}=\text{id on \ensuremath{Y} and for which }$the
correspondence $b\rightarrow a$ is universal (i.e. $a$ only depends
on $b$ but not on $r,\alpha$, see Remark 3), therefore 10.16 and
10.17 in the original proof are simultaneously satisfied for $K_{1}=\text{max}(C(r,\alpha,U,\varOmega),C(0,\gamma,U,\varOmega))$,
where the constants are provided by Theorem 3 (this fact is implicit
but not mentioned in \cite{CDK}). 

\emph{Step 2. }For any real \emph{$n\times n$ }matrix $\xi$ let
\[
Q(\xi)=\text{det}\,(I+\xi)-1-\text{trace}(\xi)
\]
where $I$ is the identity matrix. Observe that a solution to (5.1)
is given by $\varphi:=v+\text{id}$ provided $v\in C^{r+1,\alpha}(\overline{\varOmega};\mathbb{R}^{n})$
satisfies
\begin{equation}
\begin{cases}
\text{div}\,v=f-1-Q(\nabla v)\\
v=0 & \text{in }U
\end{cases}
\end{equation}
Letting $N(v)=f-1-Q(\nabla v)$, it is immediate to verify that (5.2)
is satisfied by any fixed point of the nonlinear operator $L^{-1}N:X\rightarrow X$.
By Banach's theorem, it remains to find a subset $B$ of $X$, complete
in relation to an adequate norm and such that $L^{-1}N$ maps $B$
into itself and acts there as a contraction ($N:X\rightarrow Y$ is
well defined since $N(v)=0$ in $U$ for any $v\in X$ (as both $f=1$
and $v=0$ in $U$) and $\int_{\varOmega}N(v)=0$, see \cite{CDK}).
Define $B$ as in \cite[p.200]{CDK} changing only $u=0$ in $\partial\varOmega$
to $u=0$ in $U$ and again endow it with the $C^{1,\gamma}$ norm
(note that all this is consistent since, by hypothesis, $1+\gamma\leq r+1+\alpha$
and $\varOmega$ is smooth, thus $C^{r+1,\alpha}(\overline{\varOmega})\subset C^{1,\gamma}(\overline{\varOmega})$,
see Remark 5). Then, $B$ is complete being a closed subset of the
original $B$ (by $C^{0}$ continuity). The remaining of the original
proof (including all estimates) is unchanged. 
\end{proof}

\section{\textbf{Solution when $f=1$ in a neighbourhood of a collar}}
\begin{thm}
\emph{(Moser's flow method solution) }Let $r\geq1$ be an integer
and $0\leq\alpha\leq1$. Let $\varOmega\subset\mathbb{R}^{n}$, $n\geq2$,
be a bounded connected open smooth set and $U$ a collar of $\overline{\varOmega}$.
Given $f\in C^{r,\alpha}(\overline{\varOmega})$, $f>0$ in $\overline{\varOmega}$,
satisfying
\[
\begin{cases}
\int_{\varOmega}f=\text{\emph{meas}}\,\varOmega & \text{ }\\
f=1 & \text{in \ensuremath{U}}
\end{cases}
\]
there exists $\varphi\in\text{\emph{Diff}}{}^{r,\alpha}(\overline{\varOmega},\overline{\varOmega})$
satisfying
\[
\begin{cases}
\text{\emph{det}}\,\nabla\varphi=f\\
\varphi=\text{\emph{id}} & \text{in }U
\end{cases}
\]
\end{thm}
\begin{proof}
(The proof is a trivial adaptation of that of Lemma 2 in \cite[p.9-10]{DM};
c.f. \cite[Theorem 10.7]{CDK}). Note that $\varOmega$ being smooth,
we have (1) $f\in C^{r-1,1/2}(\overline{\varOmega})$ if $\alpha=0$
and (2) $f\in C^{r,1/2}(\overline{\varOmega})$ if $\alpha=1$. By
Theorem 3 there is a solution $v$ to 
\[
\begin{cases}
\text{div\,}v=f-1 & \text{ }\\
v=0 & \text{in }\ensuremath{U}
\end{cases}
\]
which solution is (1) in $C^{r,0}(\overline{\varOmega};\mathbb{R}^{n})\supset C^{r,1/2}(\overline{\varOmega};\mathbb{R}^{n})$
if $\alpha=0$, (2) in $C^{r,1}(\overline{\varOmega};\mathbb{R}^{n})\supset C^{r+1,1/2}(\overline{\varOmega};\mathbb{R}^{n})$
if $\alpha=1$, (3) in $C^{r,\alpha}(\overline{\varOmega};\mathbb{R}^{n})\supset C^{r+1,\alpha}(\overline{\varOmega};\mathbb{R}^{n})$
if $0<\alpha<1$ (for (1) and (2) we have just used that $C^{r,0}\subset C^{r-1,1/2}$
and $C^{r,1}\subset C^{r,1/2}$, see Remark 5). Thus $v$ and $v_{t}$
(see \cite{DM}) are always of class $C^{r,\alpha}$ and so are the
solution diffeomorphisms $\Phi_{t}$, $t\in[0,1]$. With the above
$v$, the proof is the same as the original one, noting that $v=0$
in $U$ implies $v_{t}=0$ in $U$ for all $\ensuremath{t\in[0,1]}$,
which by its turn implies $\Phi_{t}=\text{id}$ in $U$, for all such
$t$. 
\end{proof}
\begin{thm}
Let $r\geq0$ be an integer and $0<\alpha<1$. Let $\varOmega\subset\mathbb{R}^{n}$,
$n\geq2$, be a bounded connected open smooth set and $U$ a collar
of $\overline{\varOmega}$. Given $f\in C^{r,\alpha}(\overline{\varOmega})$,
$f>0$ in $\overline{\varOmega}$, satisfying
\[
\begin{cases}
\int_{\varOmega}f=\text{\emph{meas}}\,\varOmega & \text{ }\\
f=1 & \text{in a neighbourhood of \ensuremath{U}}
\end{cases}
\]
there exists $\varphi\in\text{\emph{Diff}}{}^{r+1,\alpha}(\overline{\varOmega},\overline{\varOmega})$
satisfying
\[
\begin{cases}
\text{\emph{det}}\,\nabla\varphi=f\\
\varphi=\text{\emph{id}} & \text{in a neighbourhood of }U
\end{cases}
\]
\end{thm}
\begin{proof}
The adaptation of Step 4 in the proof of Theorem 1' \cite[p.12]{DM}
requires special attention since it is not straightforward (see Section
1.1). Here the general case of arbitrary $f$ is reduced to the case
where $\left\Vert f-1\right\Vert _{C^{0,\gamma}}$ is small enough
(Theorem 4), for some fixed $0<\gamma<\alpha<1$. Let $\epsilon=\epsilon(r,\alpha,\gamma,U,\varOmega)>0$
be the constant given by Theorem 4. We shall find $F\in C^{r,\alpha}(\overline{\varOmega}),$
$F>0$ in $\overline{\varOmega}$, satisfying 
\begin{equation}
\begin{cases}
\int_{\varOmega}F=\text{meas}\,\varOmega & \text{ }\\
F=1 & \text{in a neighbourhood of \ensuremath{U}}\\
\left\Vert F-1\right\Vert _{C^{0,\gamma}}\leq\epsilon
\end{cases}
\end{equation}
$F$ being the product of $f$ and  $h/\widetilde{f}$ , where $\widetilde{f}$
is a convolution of $f$ and $h$ is a measure correcting smooth function
$C^{0,\gamma}$-close to $1$. 

Fix $0<\gamma<\alpha$. Note that $f\in C^{0,\gamma}(\overline{\varOmega})$
since $\varOmega$ is smooth (and thus Lipschitz), see Remark 5. By
continuity of the multiplication and reciprocal ($1/\cdot$) operations
in relation to the $C^{0,\gamma}$ norm, there is $\delta>0$ such
that for any $\widetilde{f},\,h\in C^{0,\gamma}(\overline{\varOmega})$,
\[
\|\widetilde{f}-f\|_{C^{0,\gamma}}\,\,,\,\,\left\Vert h-1\right\Vert _{C^{0,\gamma}}\leq\delta\Longrightarrow\left\Vert \frac{hf}{\widetilde{f}}-1\right\Vert _{C^{0,\gamma}}\leq\epsilon
\]
Reparametrizing the 2nd factor of $\partial\varOmega\times[0,\infty)$
we may assume that $U=U_{1}=\zeta(\partial\varOmega\times[0,1])$
for some collar embedding $\zeta$:$\partial\varOmega\times[0,\infty)\hookrightarrow\overline{\varOmega}$
(see Definition 2), and that $f=1$ in $U_{3}$, where $U_{t}:=\zeta(\partial\varOmega\times[0,t${]}),
for each $t>0$. 

(A) \emph{Claim. There is a sequence $f_{k}\in C^{\infty}(\overline{\varOmega})$,
$k\in\mathbb{Z}^{+}$, satisfying}
\begin{enumerate}
\item $f_{k}>0$ \emph{in} $\overline{\varOmega}$ 
\item $f_{k}=1$ \emph{in} $U_{2}$
\item $\left\Vert f_{k}-f\right\Vert _{C^{0,\gamma}}\leq\delta$ 
\item $\left\Vert f_{k}-f\right\Vert _{C^{0,\gamma}}\xrightarrow[k\rightarrow\infty]{}0$
\end{enumerate}
Fix a mollifier $\rho\in C^{\infty}(\mathbb{R}^{n})$ such that $\rho>0$
in $\mathbb{B}^{n}$, $\rho=0$ elsewhere and $\int_{\mathbb{R}^{n}}\rho=1$.
Since $f$ extends by $1$ to the whole $\mathbb{R}^{n}$ (in the
$C^{r,\alpha}$class), for $k\in\mathbb{Z}^{+}$
\[
f_{k}:=\rho_{k}*f\in C^{\infty}(\overline{\varOmega})
\]
where $\rho_{k}(x)=k^{n}\rho(kx)$, is well defined and $f_{k}>0$
in $\overline{\varOmega}$ ($*$ is the convolution operator), thus
(1) holds; for $k$ large enough, say $k>k_{0}$, (2) holds since
$\text{supp}(f_{k}-1)\subset\text{supp}(f-1)+\text{supp}\,\rho_{k}$,
$\text{\text{supp}\,}\rho_{k}\rightarrow\{O\}$ (in the Hausdorff
metric, $O$ the origin of $\mathbb{R}^{n}$), $\text{supp}(f-1)\subset\overline{\varOmega\setminus U_{3}}$
and $U_{2}\subset\text{int}\,U_{3}$ (in $\overline{\varOmega})$.
To see that (4) holds, first note that $f,\,f_{k}\in C^{0,\alpha}(\overline{\varOmega})$
satisfy

(a) $\left\Vert f_{k}-f\right\Vert _{C^{0}}\xrightarrow[k\rightarrow\infty]{}0$ 

(b) $\left[f_{k}\right]{}_{C^{0,\alpha}}\leq\left[f\right]{}_{C^{0,\alpha}}$
(see Definition 1)

\noindent For this last assertion see for instance \cite[p.148]{GT},
noting that $f=1$ in a neighbourhood of $\partial\varOmega$, thus
$f$ extends by $1$ to $\widehat{f}\in C^{0,\alpha}(\mathbb{R}^{n})$
and $[\widehat{f}]{}_{C^{0,\alpha}}=\left[f\right]{}_{C^{0,\alpha}}$
(the norms being taken on the respective domains of definition). Now,
(4) easily follows from (a) and (b), see for instance Step 1.1 in
the proof of \cite[Theorem 16.22]{CDK}. Therefore, for $k$ large
enough, say $k>k_{1}\geq k_{0}$, (3) holds and reindexing $f_{k}$
as $f_{k}\rightarrow f_{k+k_{1}}$, both (2) and (3) hold for all
$k\in\mathbb{Z}^{+}$and the Claim is proved. Note that by the continuity
of the reciprocal operation in relation to the $C^{0,\gamma}$ norm,
(4) also implies 

(5) $\left\Vert f/f_{k}-1\right\Vert _{C^{0,\gamma}}\xrightarrow[k\rightarrow\infty]{}0$

\begin{figure}[t]
\noindent \begin{centering}
\includegraphics[scale=0.6]{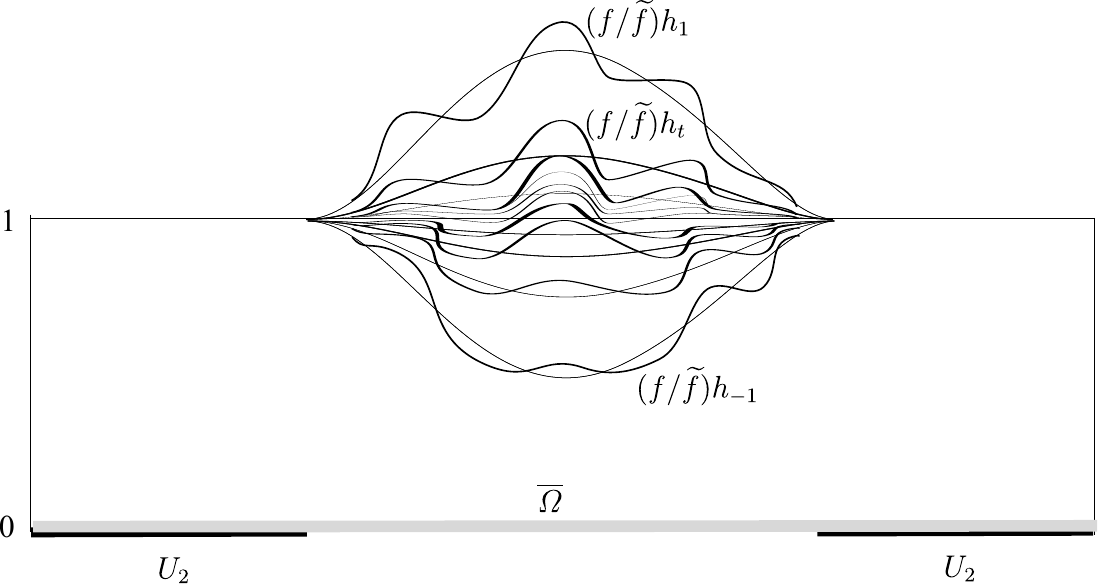}
\par\end{centering}
\caption{Finding $h_{\widehat{t}}$ satisfying $\int_{\varOmega}(f/\widetilde{f})h_{\widehat{t}}=\text{meas}\,\varOmega$.
The functions $h_{t}$ are seen in the background (bell shaped). }
\end{figure}

(B) \emph{Finding a measure correcting smooth function} $h$. Fix
$\phi\in C^{\infty}(\overline{\varOmega})$ such that $\phi=0$ in
$U_{2}$ and $0<\phi<1$ elsewhere. Take $\eta>0$ small enough so
that $H:=\eta\phi$ satisfies $\left\Vert H\right\Vert _{C^{0,\gamma}}\leq\delta$.
For $t\in[-1,1]$ let $h_{t}:=1+tH$. Note that 
\begin{itemize}
\item $h_{t}>0$ in $\overline{\varOmega}$
\item $h_{t}=1$ in $U_{2}$
\item $\left\Vert h_{t}-1\right\Vert _{C^{0,\gamma}}\leq\left\Vert H\right\Vert _{C^{0,\gamma}}\leq\delta$
\item $\int_{\varOmega}h_{1}=\text{meas}\,\varOmega+\int_{\varOmega}H>\text{meas}\,\varOmega$
\item $\int_{\varOmega}h_{-1}=\text{meas}\,\varOmega-\int_{\varOmega}H<\text{meas}\,\varOmega$
\end{itemize}
Now, by (5) above, fixing $k\in\mathbb{Z}^{+}$ large enough and letting
$\widetilde{f}:=f_{k}$ we have
\[
\int_{\varOmega}(f/\widetilde{f})h_{1}>\text{meas}\,\varOmega\text{ \,\,\,and \,\,\,}\int_{\varOmega}(f/\widetilde{f})h_{-1}<\text{meas}\,\varOmega
\]
(we can see $f/\widetilde{f}$ acting (by multiplication) as a small
$C^{0}$ perturbation on $h_{1}$ and $h_{-1}$, see Fig. 5.1). As
$\int_{\varOmega}(f/\widetilde{f})h_{t}$ varies continuously with
$t$, by the intermediate value theorem 
\[
\int_{\varOmega}(f/\widetilde{f})h=\text{meas}\,\varOmega
\]
where $h:=h_{\widehat{t}}$ for some $-1<\widehat{t}<1$. 

(C) \emph{Finding the solution diffeomorphism }$\varphi$. Summing
up, we now have 
\begin{itemize}
\item $\widetilde{f},\,h\in C^{\infty}(\overline{\varOmega})$, $\widetilde{f},\,h>0$
in $\overline{\varOmega}$
\item $\widetilde{f}=1=h$ in $U_{2}$
\item $\|\widetilde{f}-f\|_{C^{0,\gamma}}\,\,,\,\,\left\Vert h-1\right\Vert _{C^{0,\gamma}}\leq\delta$
\end{itemize}
Therefore, $F:=(h/\widetilde{f})f\in C^{r,\alpha}(\overline{\varOmega})$,
$F>0$ in $\overline{\varOmega}$ and (6.1) above holds, the neighbourhood
of $U$ in question being $U_{2}$. Now (as in the original proof
\cite[p.13]{DM}), use Theorem 4 to find a solution $\varphi_{1}\in\text{Diff}{}^{r+1,\alpha}(\overline{\varOmega},\overline{\varOmega})$
of
\[
\begin{cases}
\text{det}\,\nabla\varphi_{1}=F\\
\varphi_{1}=\text{id} & \text{in }U_{2}
\end{cases}
\]
and Theorem 5 to find $\varphi_{2}\in\text{Diff}{}^{r+1,\alpha}(\overline{\varOmega},\overline{\varOmega})$
solving
\[
\begin{cases}
\text{det}\,\nabla\varphi_{2}=(\widetilde{f}/h)\circ\varphi_{1}^{-1}\\
\varphi_{2}=\text{id} & \text{in }U_{2}
\end{cases}
\]
noting that $G:=(\widetilde{f}/h)\circ\varphi_{1}^{-1}\in C^{r+1,\alpha}(\overline{\varOmega}),$
$G>0$ in $\overline{\varOmega}$, $\int_{\varOmega}G=\text{meas}\,\varOmega$
(by the change of variables theorem) and $G=1$ in $U_{2}$. It is
immediate to verify that $\varphi=\varphi_{2}\circ\varphi_{1}$ has
all claimed properties ($U_{2}$ being the neighbourhood of $U$ where
$\varphi=\text{id}$). 
\end{proof}

\section{\textbf{The main result}}

A stronger formulation of Theorem 1 in the case $\text{supp}(f-1)\subset\varOmega$
is proved below. It shows that in this case, a solution diffeomorphism
$\varphi$ can be found that is the identity in a neighbourhood $V_{d}$
of $\partial\varOmega$ (in $\overline{\varOmega})$, which neighbourhood
depends only on the distance $d$ from $\text{supp}(f-1)$ to $\partial\varOmega$
(being independent of the function $f$ itself). When $\text{supp}(f-1)\not\subset\varOmega$,
the solution diffeomorphism is found instead via Dacorogna-Moser original
result \cite[Theorem 1']{DM}, noting that the required regularity
of the boundary can be lowered from $C^{r+3,\alpha}$ to $C^{r+2,\alpha}$
as remarked in \cite[Theorem 10.3]{CDK} (obviously, the only reason
for requiring $\varOmega$ to be of class $C^{r+2,\alpha}$ in Theorem
1 is to guarantee the existence of Dacorogna-Moser solution in the
later case).

Before stating the main result we need to establish a convention,
whose reason is explained below.

\noindent \textbf{Convention.} If $\varOmega\subset\mathbb{R}^{n}$
is a domain (see Definition 1), then $d(\emptyset,\,\partial\varOmega):=\text{inradius}\,\varOmega$,
where $d(\cdot,*)$ is the euclidean distance between two subsets
of $\mathbb{R}^{n}$ and 
\[
\text{inradius}\,\varOmega:=\text{sup}\left\{ \epsilon>0:\,\varOmega\text{ contains an open ball of radius }\epsilon\right\} 
\]

\begin{thm}
\emph{(Dacorogna-Moser Theorem - Case $\text{supp}(f-1)\subset\varOmega$).}
Let $\varOmega\subset\mathbb{R}^{n}$ be a bounded connected open
set, $r\geq0$ an integer and $0<\alpha<1$. For each $0<c\leq R:=\text{\emph{inradius}}\,\varOmega$
there exists a neighbourhood $V_{c}$ of $\partial\varOmega$ in $\overline{\varOmega}$
such that: given any $f\in C^{r,\alpha}(\overline{\varOmega})$ with
$f>0$ in $\overline{\varOmega}$ and any $0<d\leq R$ satisfying:
\[
\begin{cases}
\int_{\varOmega}f=\text{\emph{meas}}\,\varOmega\\
d(\text{\emph{supp}}(f-1),\,\partial\varOmega)\geq d
\end{cases}
\]
there exists $\varphi\in\text{\emph{Diff}}{}^{r+1,\alpha}(\overline{\varOmega},\overline{\varOmega})$
satisfying:
\[
\left\{ \begin{array}{llll}
\text{\emph{det}}\,\nabla\varphi=f\\
\varphi=\text{\emph{id}} & in\:V_{d}
\end{array}\right.
\]
\end{thm}
\begin{rem}
The construction of the neighbourhoods $V_{d}$ in the proof immediately
reveals that these are nested, 
\begin{equation}
0<\widehat{c}<c\leq R\Longrightarrow V_{\widehat{c}}\subset V_{c}\text{\,, \,\, }V_{c}\xrightarrow[c\rightarrow0^{+}]{d_{H}}\partial\varOmega\text{ \,\,\,\,and \,\,\,\,}V_{R}=\overline{\varOmega}.
\end{equation}
(as usual, $d_{H}(\cdot,*)$ is the Hausdorff metric in the space
of nonvoid closed subset of $\mathbb{R}^{n}$).
\end{rem}
\begin{rem}
If $\text{supp}(f-1)\neq\emptyset$ then this set has nonvoid interior,
thus its distance to $\partial\varOmega$ is smaller than $R:=\text{inradius}\,\varOmega$.
The compactness of $\overline{\varOmega}$ actually implies that $B_{R}(x)\subset\varOmega$
for some $x\in\varOmega$, and it is easily seen that for each $0<\epsilon\leq1$
there are functions $f_{\epsilon}$ as above for which $\text{supp}(f_{\epsilon}-1)=\overline{B_{\epsilon R}(x)}$,
thus implying that $d(\text{supp}(f_{\epsilon}-1),\partial\varOmega)=(1-\epsilon)R$
may actually attain any value $0\leq d<R$. That is why we have adopted
the convention $d(\emptyset,\partial\varOmega):=R$ thus covering
the limit case $\text{\text{supp}}(f-1)=\emptyset$ in a continuous
way (roughly speaking, $\text{\text{supp}}(f_{\epsilon}-1)$ vanishes
in the limit, as $\epsilon\rightarrow0^{+}$, since it cannot be reduced
to a single point). 
\end{rem}
\begin{proof}
(of theorem 7). If $n=1$ then the solution is trivial, see Section
2.1. Suppose that $n\geq2$. If $d=R:=\text{inradius}\,\varOmega$,
then $\text{supp}(f-1)=\emptyset$ i.e. $f=1$ in $\ensuremath{\overline{\varOmega}}$,
in which case we have the natural solution $\varphi=\text{id in \ensuremath{\overline{\varOmega}}}=:V_{R}$.
Otherwise we proceed as follows.

(A) \emph{Construction of a suitable smooth exhaustion of $\varOmega$.
}For each $0<c<R$, let $[c]$ be the unique $k\in\mathbb{Z}^{+}$
such that 
\[
\frac{R}{k+1}\leq c<\frac{R}{k}
\]
By Lemma 1 (see the Appendix), we can find a subsequence of $\varOmega_{k\in\mathbb{Z}^{+}}$,
still labeled $\varOmega_{k\in\mathbb{Z}^{+}}$, and a sequence $U_{k\in\mathbb{Z}^{+}},$
$U_{k}$ a small collar of $\overline{\varOmega_{k}}$, such that
\begin{enumerate}
\item $d_{H}\big(U_{1}\,\cup\,(\overline{\varOmega}\setminus\varOmega_{1}),\,\partial\varOmega\big)<R/2$
\item $d_{H}\big(U_{k+1}\,\cup\,(\overline{\varOmega}\setminus\varOmega_{k+1}),\,\partial\varOmega\big)<\text{min}\left(\frac{R}{k+2},\,d(\partial\varOmega_{k},\partial\varOmega)\right)$,
for each $k\geq$1.
\end{enumerate}
Define 
\[
\begin{cases}
V_{k}^{*}=U_{k}\,\cup\,(\overline{\varOmega}\setminus\varOmega_{k}) & \text{for }k\in\mathbb{Z}^{+}\\
V_{c}=V_{[c]}^{*} & \text{for }0<c<R\\
V_{R}=\overline{\varOmega}
\end{cases}
\]
Clearly, by construction $V_{k+1}^{*}\subset\text{int}\,V_{k}^{*}$
and $0<\widehat{c}<c\leq R$ implies $V_{\widehat{c}}\subset V_{c}\subset\overline{\varOmega}$,
as $0<\widehat{c}<c<R$ implies $[\widehat{c}]\geq[c]$. Also, by
construction, $d_{H}(V_{k}^{*},\partial\varOmega)$ tends to zero
as $k\rightarrow\infty$, hence $d_{H}(V_{c},\partial\varOmega)$
tends to zero as $c\rightarrow0^{+}$ (thus (7.1) holds).

(B) \emph{Solution for $0<d<R=\text{\emph{inradius}}\,\varOmega$.
}Given $f$ and $d$ as in the statement, note that, by construction
of $V_{[d]}^{*}$, $f=1$ in a neighbourhood of $V_{[d]}^{*}$, hence
$\int_{\varOmega\setminus\varOmega_{[d]}}f=\int_{\varOmega\setminus\varOmega_{[d]}}1=\text{meas}\,\varOmega\setminus\varOmega_{[d]}$,
thus
\[
\int_{\varOmega_{[d]}}f=\int_{\varOmega}f-\int_{\varOmega\setminus\varOmega_{[d]}}f=\text{meas}\,\varOmega-\text{meas}\,\varOmega\setminus\varOmega_{[d]}=\text{meas}\,\varOmega_{[d]}
\]
and $f=1$ in a neighbourhood of collar $U_{[d]}$ (in $\overline{\varOmega_{[d]}}$).
Apply Theorem 6 to get $\widehat{\varphi}\in\text{Diff}{}^{r+1,\alpha}(\overline{\varOmega_{[d]}},\overline{\varOmega_{[d]}})$
satisfying
\[
\begin{cases}
\text{det}\,\nabla\widehat{\varphi}=f|_{\overline{\varOmega_{[d]}}}\\
\widehat{\varphi}=\text{id} & \text{in }U_{[d]}
\end{cases}
\]
Finally, extend this solution to $\varphi\in\text{Diff}{}^{r+1,\alpha}(\overline{\varOmega},\overline{\varOmega})$,
letting $\varphi=\text{id}$ in $\overline{\varOmega}\setminus\varOmega_{[d]}$.
It is immediate to check that $\varphi$ has all claimed properties,
in particular $\varphi=\text{id}$ in $V_{d}=V_{[d]}^{*}=U_{[d]}\,\cup\,\overline{\varOmega}\setminus\varOmega_{[d]}$.
\end{proof}

\section{\textbf{Impossibility of the estimate} \textup{$\left\Vert \varphi-\text{id}\right\Vert _{C^{r+1,\alpha}}\leq C\left\Vert f-1\right\Vert _{C^{r,\alpha}}$}}

By the present method, which guarantees control of the support of
solutions, an estimate in Theorem 7 as that in Theorem 10.3 of \cite[p.192]{CDK}
cannot be attained. Even restricting to functions $f$ that are $C^{0,\alpha}$-close
enough to 1 (c.f. Theorem 4 and \cite[Theorem 10.9]{CDK}), we shall
exhibit strong evidence pointing to the fact that the estimate
\begin{equation}
\left\Vert \varphi-\text{id}\right\Vert _{C^{r+1,\alpha}}\leq C\left\Vert f-1\right\Vert _{C^{r,\alpha}}
\end{equation}
for some constant $C=C(r,\alpha,\varOmega)>0$, is impossible to obtain
simultaneously with control of support
\[
\ensuremath{\text{supp}(f-1)\subset\varOmega}\Longrightarrow\text{supp}(\varphi-\text{id})\subset\varOmega
\]
if one employs the present method (or variants of it) for the construction
of the solutions. To simplify the explanation, suppose that $\varOmega$
is smooth. Obviously, any solution to $\text{det}\,\nabla\varphi=f$
in $\varOmega$ satisfies
\[
\text{supp}(\varphi-\text{id})\supset\text{supp}(f-1)
\]
Actually, if $\text{supp}(f-1)\subset\varOmega$, then to carry out
the present method for the construction of a solution $\varphi$ satisfying
$\text{supp}(\varphi-\text{id})\subset\varOmega$, we need in first
place to fix a collar $U=\zeta(\partial\varOmega\times[0,\epsilon])$
of $\overline{\varOmega}$ such that $U\subset\overline{\varOmega}\setminus\text{supp}(f-1),$
in order to be able to apply Theorem 4. Clearly, 
\[
\text{thick}\,U<d\left(\text{supp}(f-1),\partial\varOmega\right)
\]
where $\text{thick}\,U$ (the \emph{thickness of $U$}) is the distance
between $\partial\varOmega$ and the ``internal'' boundary $\partial_{0}U$
of $U$ i.e.
\[
\text{thick}\,U:=d(\partial\varOmega,\partial_{0}U)\text{ \,\,where \,\,}\partial_{0}U:=\zeta(\partial\varOmega\times\epsilon)=\partial U\setminus\partial\varOmega
\]
Let $U_{k\in\mathbb{Z}^{+}}$ be a sequence of collars of $\overline{\varOmega}$
such that 
\[
\text{thick}\,U_{k}\xrightarrow[k\rightarrow\infty]{}0
\]
We shall now produce enough evidence that in Theorem 4, fixing $\gamma=\alpha$,
the undesirable facts
\[
c=c(r,\alpha,U_{k},\varOmega)\xrightarrow[k\rightarrow\infty]{}\infty\text{ \,\,and \,\,}\epsilon=\epsilon(r,\alpha,U_{k},\varOmega)\xrightarrow[k\rightarrow\infty]{}0
\]
cannot be avoided, thus dissipating any hope of establishing the estimate
(8.1) with $C$ independent of $d\left(\text{supp}(f-1),\partial\varOmega\right)$.
The proof of Theorem 4 is closely modeled on that of \cite[Theorem 10.9]{CDK},
all the estimates being the same. In particular, a brief inspection
reveals that 
\[
c=c(r,\alpha,U_{k},\varOmega)=2K_{1}\text{ \,\,and \,\,}\epsilon=\epsilon(r,\alpha,U_{k},\varOmega)\leq1/2K_{1}
\]
where $K_{1}\geq C$, $C=C(r,\alpha,U_{k},\varOmega)>0$ being the
constant given by Theorem 3 (see Step 1 in the Proof of Theorem 4).
Actually, in this $C$ it enters as a multiplicative factor a constant
$\widehat{C}=C_{2}(r,U_{k})\geq$1, controlling (in the case under
question) the $C^{r+2,\alpha}$ norm of an extension operator introduced
in the Proof of Theorem 2. We now show evidence that, when $r\geq1$,
there is no way to keep $C_{2}(r,U_{k})$ bounded as $k\rightarrow\infty$
i.e. when $k\rightarrow\infty$ the thickness of $U_{k}$ tends to
zero and 

\begin{figure}[t]
\noindent \begin{centering}
\includegraphics[scale=0.2]{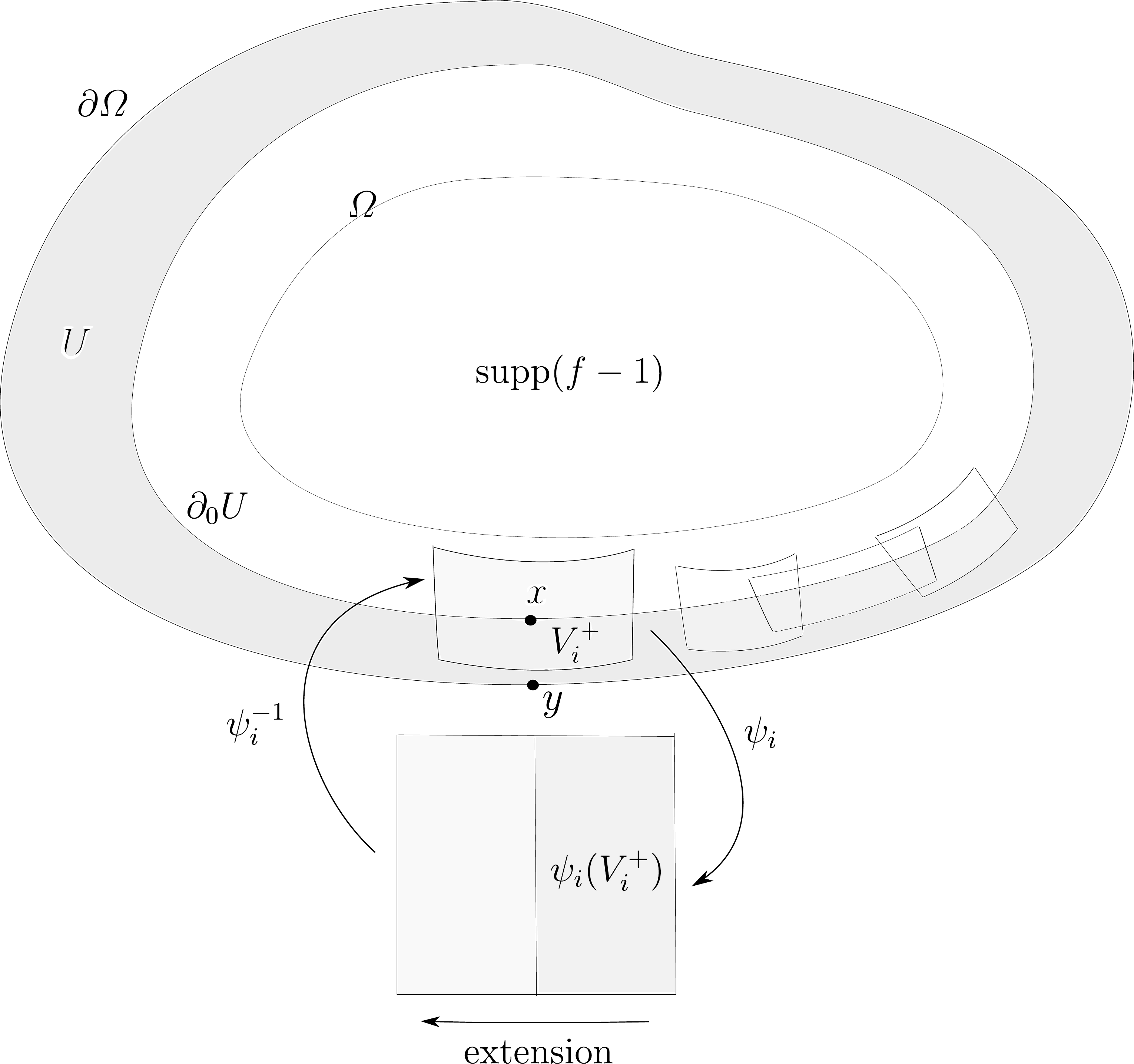}
\par\end{centering}
\caption{Extending $g\in C^{1}(U)$ to the whole $\overline{\varOmega}.$}
\end{figure}

\begin{equation}
C_{2}(r,U_{k})\xrightarrow[k\rightarrow\infty]{}\infty
\end{equation}
cannot be avoided. All extension methods (applicable in the context
of (C) in the Proof of Theorem 2) that are known to us are variants
of the same global strategy and are consequently subject to problem
(8.2), see below. Let $U$ be any of the $U_{k}$'s. Briefly, the
extension of a function $g\in C^{1}(U)$ to the whole $\overline{\varOmega}$
is obtained gluing together finitely many local extensions from the
interior to the exterior of $U$, performed on small balls (or cubes)
centred at points of $\partial_{0}U$. To simplify the explanation,
we adopt the extension method described in \cite[p.136]{GT} with
the modification of Remark 2. Since we wish only to extend $g$ to
the whole $\overline{\varOmega}$ (and not to the whole $\mathbb{R}^{n}$),
we start by covering $\partial_{0}U$ with finitely many open ``cubes''
$V_{i}$, $1\leq i\leq j$, each having one of its halves $V_{i}^{+}$
contained in $U$ (see Fig. 7.1).\footnote{More precisely, $\cup_{1\leq i\leq j}V_{i}\supset\partial_{0}U$,
where each $V_{i}$ is an open set intersecting $\partial_{0}U$ for
which there is a (boundary rectifying) diffeomorphism $\psi_{i}\in\text{Diff}^{\infty}(V_{i},\,(-2,2)^{n})$
such that $\psi_{i}(V_{i}^{+})=[0,2)\times(-2,2)^{n-1}$, where $V_{i}^{+}=V_{i}\cap U$,
thus implying $V_{i}\cap\partial\varOmega=\emptyset$. }As explained in \cite{GT} (see also Remark 2), we then use the boundary
rectifying diffeomorphisms (see Footnote 4) and Seeley's extension
operator to get, for each $1\leq i\leq j$, a $C^{1}$ extension $g_{i}$
of $g|_{V_{i}^{+}}$ to the whole $V_{i}$. Note that if we set $V_{0}:=U\setminus\partial_{0}U$,
then $\{V_{i}\}_{0\leq i\leq j}$ is an open covering of $U$ in $\overline{\varOmega}.$
Fix a partition of unity $\{\eta_{i}\}_{0\leq i\leq j}$ subordinate
to this covering.\footnote{Note that necessarily $\eta_{0}=0$ in a neighbourhood of $\partial_{0}U$
and $\eta_{0}=1$ in a neighbourhood of $\partial\varOmega$ (in $\overline{\varOmega})$.} Finally, let $g_{0}=g|_{V_{0}}$ and define the desired extension
$\widetilde{g}\in C^{1}(\overline{\varOmega})$ of $g$ as 
\[
\widetilde{g}=\sum_{i=0}^{j}\eta_{i}g_{i}
\]
(with the convention that $g_{i}=0$ in $\overline{\varOmega}\setminus V_{i}$).
Now, in order to to find a constant $\widehat{C}=\widehat{C}(1,U)\geq$1
such that 
\[
\|\widetilde{g}\|_{C^{1}(\overline{\varOmega})}\leq\widehat{C}\|g\|_{C^{1}(U)}
\]
we are naturally lead to the estimates 

\[
\begin{array}{lll}
\|\widetilde{g}\|_{C^{1}(\overline{\varOmega})} & \leq & \sum_{i=0}^{j}\|\eta_{i}g_{i}\|_{C^{1}}\\
 & \leq & \sum_{i=0}^{j}\|\eta_{i}\|_{C^{1}}\|g_{i}\|_{C^{0}}+\|\eta_{i}\|_{C^{0}}\|g_{i}\|_{C^{1}}\\
 & \leq & 2\sum_{i=0}^{j}\|\eta_{i}\|_{C^{1}}\|g_{i}\|_{C^{1}}
\end{array}
\]
It is easily seen that there is a constant $K\geq1$ depending only
on the boundary rectifying diffeomorphisms $\psi_{i}$ (see Footnote
4) and on the extension operator from the right halfcube $\mathscr{C}^{+}$
to the cube $\mathscr{C}=(-2,2)^{n}$ (in our case Seeley's operator)
such that 
\[
\|g_{i}\|_{C^{1}}\leq K\|g\|_{C^{1}(U)}
\]
and therefore the above estimate leads to
\[
\|\widetilde{g}\|_{C^{1}(\overline{\varOmega})}\leq\widehat{C}\|g\|_{C^{1}(U)}\text{ \,\,\,where \,\,\,}\widehat{C}=2K\sum_{i=0}^{j}\|\eta_{i}\|_{C^{1}}
\]
But the problem now lies in the factor $\sum_{i=0}^{j}\|\eta_{i}\|_{C^{1}},$
for one can easily see that as $k\rightarrow\infty$, the thickness
of $U=U_{k}$ tends to zero and simultaneously this quantity diverges
to $\infty$, regardless of the particular partition of unity employed
for each $U_{k}$. In fact let $\delta=\text{thick}\,U$ and fix $x\in\partial_{0}U$
and $y\in\partial\varOmega$ such that $d(x,y)=|x-y|=\delta$ (see
Fig. 7.1). For $1\leq i\leq j$ , $y\notin V_{i}$ since $V_{i}\cap\partial\varOmega=\emptyset$
(see Footnote 5), hence for all such indices $i$, $\eta_{i}(y)=0$.
Therefore, by the mean value theorem, for $1\leq i\leq j$
\[
\|\eta_{i}\|_{C^{1}}\geq\frac{\eta_{i}(x)-\eta_{i}(y)}{|x-y|}=\frac{\eta_{i}(x)}{\delta}
\]
and since $\sum_{i=1}^{j}\eta_{i}(x)=1$ (as $x\in\partial_{0}U$
implies $\eta_{0}(x)=0$) we finally have
\[
\sum_{i=0}^{j}\|\eta_{i}\|_{C^{1}}\geq\delta^{-1}\sum_{i=1}^{j}\eta_{i}(x)=\delta^{-1}
\]
therefore, when $k\rightarrow\infty$, $\widehat{C}=\widehat{C}(1,U_{k})\rightarrow\infty$
as $\delta_{k}:=\text{thick}\,U_{k}\rightarrow0$. It remains to mention
that if instead of the above extension operator, that of \cite[p.342]{CDK}
is employed, then we run into the very same problem (see in particular
\cite[p.353-355]{CDK} where an explicit formula for this extension
operator is obtained, noting that the auxiliary functions $\lambda_{i}$
play the analogue role to the above $\eta_{i}$).

\section{\textbf{Appendix }}

For the sake of completeness we include a brief proof that any domain
has an exhaustion by smooth domains (despite being a well known fact,
we could not locate a proof in the literature).
\begin{defn}
For each $\mathfrak{M}\subset2^{\mathbb{R}^{n}}$ ($2^{\mathbb{R}^{n}}$
being the set of all subsets of $\mathbb{R}^{n}$), we define $\cup\mathfrak{M}=\cup_{S\in\mathfrak{M}}S\subset\mathbb{R}^{n}$.
\end{defn}
\begin{lem}
\emph{(Exhaustion by smooth domains).} Let $\varOmega\subset\mathbb{R}^{n}$
be a bounded connected open set. Then there is a sequence $\varOmega_{k}$
of bounded connected open smooth sets satisfying:
\begin{enumerate}
\item $\overline{\varOmega_{k}}\subset\varOmega_{k+1}\subset\varOmega$
~for all $k\in\mathbb{Z}^{+}$
\item $\cup_{k\in\mathbb{Z}^{+}}\varOmega_{k}=\varOmega$
\item $d_{H}(\overline{\varOmega}\setminus\varOmega_{k},\partial\varOmega)\xrightarrow[k\rightarrow\infty]{}0$
\end{enumerate}
\end{lem}
\begin{proof}
Let $x_{k\in\mathbb{Z}^{+}}$ be a dense sequence in $\varOmega$
and define $B_{k}=B(x_{k},\delta_{k})$, where $\delta_{k}=d(x_{k},\partial\varOmega)/2$.
Let $\mathscr{B}=\{B_{k}\}_{k\in\mathbb{Z}^{+}}$. Clearly
\begin{equation}
\cup\mathscr{B}=\varOmega
\end{equation}
Let $\varOmega_{1}=B_{1}$. We proceed by induction over $k\in\mathbb{Z}^{+}$.
Supposing that $\varOmega_{k}$ has already been found, we now find
$\varOmega_{k+1}$. Define
\[
\xi(k)=\text{min}\,\left\{ j\in\mathbb{Z}^{+}:\,\cup_{i=1}^{j}B_{i}\supset\overline{\varOmega_{k}}\right\} 
\]
\[
\varTheta_{k}=\{B_{i}:\,1\leq i\leq\xi(k)\}
\]
Eventually, $\cup\varTheta_{k}$ is disconnected. In this case let
$\gamma_{k}:[0,1]\rightarrow\varOmega$ be an injective path that
intersects all the components of $\cup\varTheta_{k}$. Since $\widetilde{\gamma_{k}}=\text{im}\,\gamma_{k}$
is compact, we can find a finite collection $\Upsilon_{k}\subset\mathscr{B}$
such that each ball in $\Upsilon_{k}$ intersects $\widetilde{\gamma_{k}}$
and $\widetilde{\gamma_{k}}\subset\cup\Upsilon_{k}$. If $\cup\varTheta_{k}$
is connected, simply let $\Upsilon_{k}=\emptyset$. Let $\mathscr{C}_{k}=\varTheta_{k}\cup\Upsilon_{k}$.
Note that $\cup$$\mathscr{C}_{k}$ is connected. Slightly increasing
the radius of each ball $B_{i}\in\mathscr{C}_{k}$ (always to less
than the double of the original radius, in order to keep its closure
within $\varOmega$), we make all the boundary spheres $\partial B_{i}$
intersect transversely, so that the union of these enlarged balls
is a connected open set $\varOmega'_{k+1}$with piecewise smooth boundary
containing $\overline{\cup\mathscr{C}_{k}}$, whose closure $\overline{\varOmega'_{k+1}}$
is contained in $\varOmega$. Smooth out the edges of $\overline{\varOmega'_{k+1}}$
to get $\overline{\varOmega{}_{k+1}}$, the closure of a connected
open smooth set $\varOmega{}_{k+1}$ satisfying

(a) $\cup\mathscr{C}_{k}\subset\varOmega{}_{k+1}$

(b) $\overline{\varOmega{}_{k+1}}\subset\varOmega$

\noindent This is clearly possible since the smoothing of $\overline{\varOmega'_{k+1}}$
can be performed arbitrarily near $\partial\varOmega'_{k+1}$. Since
$\overline{\varOmega_{k}}\subset\cup\mathscr{C}_{k}$, it follows
from (a) and (b) that (1) holds. Observe that the inductively defined
function $\xi:\mathbb{Z}^{+}\rightarrow\mathbb{Z}^{+}$ is strictly
increasing, as for all $k\geq1$, $\overline{\varOmega_{k}}\subset\varTheta_{k}\subset\varOmega{}_{k+1}$,
therefore by (9.1), (2) holds. To see that (3) holds we proceed by
contradiction. First note that $\partial\varOmega\subset\overline{\varOmega}\setminus\varOmega_{k}$
for all $k\geq1$, hence if (3) fails then there is $\epsilon>0$,
a subsequence of $\varOmega_{k}$, still labeled $\varOmega_{k}$,
and a sequence of points $z_{k}\in\overline{\varOmega}\setminus\varOmega_{k}$
such that $d(z_{k},\partial\varOmega)\geq\epsilon$. Since $\overline{\varOmega}$
is compact, $z_{k}$ has a subsequence (still labeled $z_{k}$) converging
to some point $z\in\varOmega$. Clearly $z$ can belong to no $\varOmega_{k}$,
otherwise, by (1), there is a neighbourhood of $z_{k}$ contained
in $\cap_{i\geq k}\varOmega_{i}$, which contradicts the existence
of sequence $z_{k}$ as defined above. But $z\notin\cup_{k\in\mathbb{Z}^{+}}\varOmega_{k}$
contradicts $\cup_{k\in\mathbb{Z}^{+}}\varOmega_{k}=\varOmega$. 
\end{proof}
\textbf{Acknowledgment.} The author wishes to thank an anonymous referee
for calling his attention to \cite[Theorem 17.3]{CDK} and to its
usefulness in the proof of Theorem 2.

\emph{Note added in proof. }After the conclusion of this work, Olivier
Kneuss informed (without further details) the author that he had also
obtained Theorem 1.

\medskip{}

\end{document}